\documentclass{conm-p-l}

\usepackage{amssymb}

\usepackage{graphicx}

\usepackage{amsmath}

\newtheorem{theorem}{Theorem}[section]

\newtheorem{prop}{Proposition}
\newtheorem{theo}[prop]{Theorem}
\newtheorem{coro}[prop]{Corollary}

\newtheorem{lemm}[prop]{Lemma}

\theoremstyle{definition}

\theoremstyle{remark}

\newtheorem{rema}[prop]{Remark}

\numberwithin{equation}{section}

\newcommand{\ra}{\rightarrow}

\newcommand{\bC}{{\mathbb C}}
\newcommand{\bD}{{\mathbb D}}

\newcommand{\bG}{{\mathbb G}}

\newcommand{\bP}{{\mathbb P}}

\newcommand{\bZ}{{\mathbb Z}}

\newcommand{\cC}{{\mathcal C}}
\newcommand{\cF}{{\mathcal F}}
\newcommand{\cO}{{\mathcal O}}
\newcommand{\cX}{{\mathcal X}}
\newcommand{\cY}{{\mathcal Y}}
\newcommand{\cZ}{{\mathcal Z}}

\newcommand{\md}{{\mathfrak d}}

\newcommand{\sg}{{\mathsf g}}

\newcommand{\Bl}{\mathrm{Bl}}
\newcommand{\End}{\mathrm{End}}
\newcommand{\GL}{\mathrm{GL}}
\newcommand{\Gr}{\mathrm{Gr}}
\newcommand{\Hom}{\mathrm{Hom}}
\newcommand{\IJJ}{\mathrm{IJ}}
\newcommand{\JJ}{\mathrm{J}}
\newcommand{\oSect}{\overline{\mathrm{Sect}}}
\newcommand{\PGL}{\mathrm{PGL}}
\newcommand{\Pic}{\mathrm{Pic}}
\newcommand{\rank}{\mathrm{rank}}
\newcommand{\Res}{\mathrm{Res}}
\newcommand{\Sect}{\mathrm{Sect}}
\newcommand{\Sym}{\mathrm{Sym}}

\begin{document}

\title[Spaces of sections]{Spaces of sections of quadric surface fibrations over curves}


\author{Brendan Hassett}
\address{Department of Mathematics\\
Rice University, MS 136 \\
Houston, Texas  77251-1892 \\
USA}
\email{hassett@rice.edu}
\thanks{The first author was supported by National Science Foundation Grants 0901645 and 0968349.}

\author{Yuri Tschinkel}
\address{Courant Institute\\
                New York University \\
                New York, NY 10012 \\
                USA }
\email{tschinkel@cims.nyu.edu}
\thanks{The second author was supported by National Science Foundation
Grants 0739380, 0901777, and 0968318.}


\date{November 3, 2011}

\begin{abstract}
We consider quadric surface fibrations over curves, defined over algebraically
closed and finite fields.  Our goal is to understand, in geometric terms,
spaces of sections for such fibrations.  We analyze varieties of
maximal isotropic subspaces in the fibers as $\bP^1$-bundles over the
discriminant double cover.  When the $\bP^1$-bundle is suitably stable,
we deduce effective estimates for the heights of sections over finite
fields satisfying
various approximation conditions.  We also discuss the behavior 
of the spaces of sections as the base of the fibration acquires 
singularities.  

\end{abstract}

\maketitle

\section{Introduction}

Let $k$ be a field of characteristic not equal to two,
$B$ a smooth projective curve of genus $\sg(B)$ over $k$, and $F$ its function field.   
A {\em quadric hypersurface fibration} is a flat projective morphism $\pi:\cX \ra B$
such that each geometric fiber is a quadric hypersurface with at worst an isolated singularity
and the generic fiber is smooth.  
Sections $\sigma:B\ra \cX$ of $\pi$ are in bijection with rational
points $X(F)$.  

Our study is motivated by arithmetic applications and analogies between 
function fields of curves and number fields.  When $k$ is a finite field, 
the following problems have been studied by various research groups:
\begin{enumerate}
\item existence of rational points, see, e.g.,  \cite{ct-kahn}, \cite{CTSD};
\item bounding the smallest height of a rational point;
\item weak approximation  \cite{Harder} and its effective versions;
\item asymptotic distribution of rational points with respect to heights, e.g., \cite{lai}, 
\cite{peyre}, \cite{bourqui}.
\end{enumerate}
All of these questions ultimately rely on algebro-geometric properties of 
spaces of sections.  In this paper we study in detail these spaces.
We relate computable invariants of quadric surfaces over function fields
of curves, like the discriminant, to geometric invariants of spaces of sections
such as the maximally rationally connected quotients of the section spaces. 

In general, spaces of rational curves on rationally connected threefolds have 
intricate geometry, even for cubic threefolds or complete intersections of two
quadrics in $\bP^5$
(see e.g., \cite{castravet}, \cite{HRS,HRS2}). 
Rational surface fibrations over $\bP^1$
appear to be much easier. 
In our case, the spaces of sections turn out to be projective bundles over
the Jacobian of the discriminant curve.  This allows us to answer the questions
above. 

The geometry of the degenerations of spaces of sections serves as a prototype for 
investigations of sections of more complicated rational surface fibrations
over curves.  However, even in the case of quadric surfaces, our inductive approach
has implications for enumerative geometry, e.g., the Gromov-Witten invariants
associated with sections of height passing through prescribed points and 
curves in the fibers.  We expect an inductive formula for this, expressed in terms of
the numerical invariants of the fibration.

\

We summarize the contents of this paper.  Section~\ref{sect:QFDH}
develops general notions of height and discriminant for quadric hypersurface
fibrations.  Section~\ref{sect:reduction}
presents the key construction of reduction to the discriminant.  This
is fundamental and well-known in the algebraic study of quadratic
forms, but here we recast it in geometric terms.  We also include
numerical estimates on the dimension of spaces of sections, from 
various points of view.  
In Section~\ref{sect:casebycase}, we show how our classification
techniques apply over $\bP^1$ and offer explicit equations for
the quadric surface fibrations in each case.  Quadric surface fibrations
admit numerous birational modifications;  the structure of these
is indicated in Section~\ref{sect:hecke}.  
We review some general facts about stability of bundles in Section~\ref{sect:SBC}.
Sections~\ref{sect:projbundle} and \ref{sect:Neron} demonstrate
how spaces of sections of quadric fibrations specialize as the 
discriminant curve acquires nodes.  We expect there exists a 
compactification for the space of quadric surface fibrations 
with square-free discriminant over the moduli space of admissible
discriminant covers (cf.~\cite{HM}),
sharing many properties with Pandharipande's
compactification of the moduli space of vector bundles over 
moduli space of stable curves \cite{Pandharipande}.
The theory of N\'eron models in the context of limiting mixed Hodge
structures offers a useful framework for the analysis of components
of the space of sections as the discriminant breaks.  
Section~\ref{sect:SDC} and \ref{sect:AA} are devoted to arithmetic
applications, e.g., effective weak approximation,
which entail effective estimates of vanishing of cohomology.

\

{\bf Acknowledgments:} 
We are grateful to A. Auel, M. Kerr, and J. Starr for helpful conversations,
and to N. Hoffmann for comments on a draft of this manuscript.  

\section{Quadratic forms, discriminants, and heights}
\label{sect:QFDH}

Let $\pi:\cX \ra B$ be a quadric hypersurface fibration of relative
dimension $n$, as defined
in the introduction.  
Let $\omega_{\pi}$ denote the relative dualizing sheaf,
an invertible sheaf that commutes with basechange.  
The {\em height} of $\cX$ is defined as
$$h(\cX)=-\deg(c_1(\omega^{-1}_{\pi})^{n+1}).$$
Note that
\begin{itemize}
\item
If $\cX \ra B$ is trivial, i.e., $\cX\simeq \cX_b \times B$ for some smooth quadric $\cX_b$,
then $h(\cX)=0$.  
\item
If $B' \ra B$ is a finite morphism of smooth projective curves
then 
$$h(\cX \times_B B')=\deg(B'/B)h(\cX).$$
\item
Every smooth quadric fibration $\cX\ra B$ also has $h(\cX)=0$.
\end{itemize}
To deduce the last statement, it suffices to observe that a smooth quadric fibration
may be trivialized after a finite flat base change $B' \ra B$.  

We define the {\em height} of a section $\sigma:B\ra \cX$ of $\pi$ to be
$$h_{\omega^{-1}_{\pi}}(\sigma)=\deg(\sigma^*\omega^{-1}_{\pi}).$$
If $\cX$ is smooth then this equals the degree of the normal bundle
$N_{\sigma}$.
We are interested in spaces of sections 
$$\Sect(\cX/B,h)=\{\sigma:B \ra \cX: h_{\omega^{-1}_{\pi}}(\sigma)=h \}.$$ 

If $k$ is algebraically closed then the
Brauer group of $k(B)$ for any smooth curve $B$ is trivial, thus
there exists a line bundle $H$ on 
$\cX$ restricting to the hyperplane class on each fiber of $\pi$.  
The sheaf $\pi_*H$ is locally free of rank $n+2$ and we have
an embedding $\cX \hookrightarrow \bP((\pi_*H)^{\vee})$.  
The defining equation is given by a section
$$q\in \Sym^2((\pi_*H))\otimes I,$$
where $I$ is an invertible sheaf of $B$.  Note that $H$ and 
$I$ can be rescaled;  for each invertible
sheaf $L$ on $B$, we may replace $H$ by $H\otimes L$ and 
$I$ by $I \otimes L^{\otimes 2}$ without altering $q$.  
Therefore, we will often normalize $H$ so that $\deg(I)=0$ or $1$;
when using this convention, we write $E=(\pi_*H)^{\vee}$.  
The parity 
$$\epsilon(\pi):=\deg(I)\pmod{2}$$
is an invariant of the fibration $\pi:\cX \ra B$.  

We may interpret the defining quadratic form $q$ as a homomorphism
$$q:E \ra E^{\vee} \otimes I,$$
self-dual under the application of $\Hom(-,I)$.  The {\em discriminant}
$\md$ is defined as the divisor where $q$ 
drops rank, which gives \cite{HT}
$$
\Delta=\deg(\md)=\deg(E^{\vee} \otimes I)-\deg(E)=
-2\deg(E)+(n+2)\deg(I),$$
so in particular
$$
\Delta \equiv \begin{cases} -2\deg(E)\pmod{2(n+2)} & \text{ if } \epsilon(\pi)\equiv 0\pmod{2} \\
                            -2\deg(E)+n+2 \pmod{2(n+2)} & \text{ if } \epsilon(\pi)\equiv 1\pmod{2}.
		\end{cases}
$$
The fibration $\pi:\cX \ra B$ has {\em square-free discriminant} if the divisor
$\md$ is reduced; a local computation shows
this is equivalent to the total space $\cX$ being smooth.  

\begin{prop} \label{prop:quadricformulaB}
If $\pi:\cX \ra B$ is a quadric hypersurface fibration of relative dimension $n$
with square-free discriminant then
$$h(\cX)=n^n \Delta.$$
\end{prop}

\begin{proof}
Let $C \ra B$ be a simply branched double cover whose branch locus contains the discriminant.
As we have seen, pulling back to $C$ increases the height:
$$h(\cX \times_{B} C)
=2h(\cX).$$

Let $x_1,\ldots,x_{\Delta}$ denote the singularities of the fibers of $\cX\times_B C\ra C$.  
We have a modification
$$\begin{array}{ccc}
\tilde{\cY} & \ra & \cY \\
\downarrow                     &     &     \\
\cX \times_B C           &  	&
\end{array}
$$
obtained by blowing up the $x_j$ and then blowing down the proper transforms
of the fibers 
of $\cX \times_B C\ra C$ containing these points.  
The resulting $\cY \ra C$ is a smooth quadric fibration.  

Let $E_1,\ldots,E_{\Delta}$ be the exceptional divisors of 
$\beta:\tilde{\cY} \ra \cX \times_B C$ over the ordinary
singularities $x_1,\ldots,x_{\Delta}$;  
in particular, $E_j$ is a smooth quadric of dimension $n$ and
$E_j^{n+1}=(-1)^n 2$.  
The discrepancy formula
$$\omega_{\tilde{\cY}}=\beta^*\omega_{\cX\times_BC}+(n-1)\sum_{j=1}^{\Delta}E_j$$
implies
$$c_1(\omega_{\tilde{\cY}/C})^{n+1}=c_1(\omega_{\cX\times_{B}C/C})^{n+1} + 
\Delta(n-1)^{n+1} (-1)^n 2.$$

On the other hand, let $F_1,\ldots,F_{\Delta}$ denote the exceptional divisors of
$\gamma:\tilde{\cY} \ra \cY$;
$F_j \simeq \bP(\cO_Q \oplus \cO_Q(1))$, where $Q$ is a smooth
quadric of dimension $n-1$ (two points when $n=1$).  
It follows that 
$$\gamma^*c_1(\omega_{\cY/C})^r\cdot F_j^{n+1-r}=(-1)^{r+1}n^r 2,$$
for $r=0,\ldots,n-1$;  we get zero when $r\ge n$.  
Here the discrepancy formula is
$$\omega_{\tilde{\cY}}=\gamma^*\omega_{\cY}+\sum_{j=1}^{\Delta}F_j.$$
Thus we find
$$\begin{array}{rcl}
c_1(\omega_{\tilde{\cY}/C})^{n+1} &=& c_1(\omega_{\cY/C})^{n+1}+\Delta \sum_{r=0}^{n-1} \binom{n+1}{r}(-1)^{r+1} n^r 2 \\
				&=& c_1(\omega_{\cY/C})^{n+1}-2\Delta \sum_{r=0}^{n-1}\binom{n+1}{r}(-n)^r \\
				&=& c_1(\omega_{\cY/C})^{n+1}-2\Delta [ (1-n)^{n+1}-((-n)^{n+1}+(n+1)(-n)^n)] \\
				&=& c_1(\omega_{\cY/C})^{n+1}-2\Delta((1-n)^{n+1}-(-n)^n)
\end{array}
$$
Note that $c_1(\omega_{\cY/C})^{n+1}=0$ as $\cY \ra C$ is smooth.  
Combining the results of our discrepancy computations, 
we obtain
$$\begin{array}{rcl}
c_1(\omega_{\cX\times_BC/C})^{n+1}&=&2\Delta((n-1)^{n+1}(-1)^{n+1}-(1-n)^{n+1}+(-n)^n) \\
				  &=&2\Delta(-n)^n,
\end{array}
$$
which yields our formula.
\end{proof}

\section{Reduction to the discriminant for quadric surface fibrations}
\label{sect:reduction}
We recall the standard argument of `reduction to the discriminant',
in geometric terms.  Let $\cX \ra B$ be a quadric surface fibration 
with square-free discriminant and generic fiber $X$.
These fibrations were studied by Bhosle \cite{BhosleMA}, especially when 
$B=\bP^1$.

\subsection*{The basic construction}
Let $\cF:=F_1(\cX) \ra B$ denote the space of lines in fibers of $\pi$;
its Stein factorization
$$\cF \ra C \stackrel{g}{\ra} B$$
is the composition of a {\em smooth} $\bP^1$-bundle and a double
cover branched along the discriminant divisor $\md$.
Let $\iota:C \ra C$ denote the covering involution.

Each section of $\pi:\cX \ra B$ yields a section of 
$\cF \ra C$ and vice versa.  Indeed, for $\sigma:B \ra \cX$ consider
the pair of lines containing $\sigma(B)$, which is a section of 
$\cF \ra C$.  Conversely, for each section $\tau:C \ra \cF$
we can take the intersection of lines
$$\ell_{\tau(c)} \cap \ell_{\tau(\iota(c))} \in  \cX_{g(c)},$$
which is a section.
Note that the universal
line over $\cF$ is a double cover of $\cX$.  

\subsection*{Reversing the construction}
Suppose that $g:C \ra B$ is a flat morphism of smooth projective curves of degree two;
we assumed the characteristic is different from two, so $g$ is tamely ramified over
a divisor $\md \subset B$. 
Fix a $\bP^1$-bundle $\cF \ra C$, which can be expressed as the projectivization of a 
vector bundle.  Restriction of scalars (Weil restriction) gives a projective morphism
$$\varpi:\Res_{C/B}(\cF) \ra B;$$
this can be interpreted as the Hilbert scheme of length-two punctual subschemes
of fibers of $\cF \ra C$.  Thus for $b\in (B\setminus \md)(\bar{k})$ we have
$$\varpi^{-1}(b)= \cF_{c_1} \times \cF_{c_2}\simeq  \bP^1 \times \bP^1, \quad g^{-1}(b)=\{c_1,c_2\},$$
geometrically a smooth quadric surface.  

Over points of the branch divisor $b\in \md\subset B$, the fiber 
$\varpi^{-1}(b)$ is set-theoretically $\Sym^2(\cF_b)\simeq \bP^2$, but non-reduced of multiplicity two.  
However, the restriction of scalars can be modified as follows:
$$\begin{array}{rcccl}
    &     &   \tilde{\cX} & & \\
    & \stackrel{\beta}{\swarrow} &        &\stackrel{\gamma}{\searrow} & \\
\Res_{C/B}(\cF) &   &    &   &   \cX \\
    & \stackrel{\varpi}{\searrow} &        &\stackrel{\pi}{\swarrow} & \\
    &     &   B & &
\end{array}
$$
where the arrows have the following definitions:
\begin{itemize}
\item  $\beta$ is obtained by blowing up the
diagonal in $\Sym^2(\cF_b)$ over each point $b\in \md$;
\item $\gamma$ is obtained by blowing down the proper transform
of $\varpi^{-1}(b)$ in $\tilde{\cX}$ over each point $b\in \md$;
\item $\pi$ is the induced morphism back to $B$.
\end{itemize}
A local computation over each $b \in \md$ shows that the fiber $\cX_b$ is isomorphic
to a quadric surface with isolated singularity.

\subsection*{Riemann-Roch computations}
Regard the space of sections $\Sect(\cX/B,h)$ as an open subscheme of the Hilbert scheme of $\cX$.
Its tangent space at $\sigma:B\ra \cX$ is
$$T_{[\sigma]}\Sect(\cX/B,h)=\Gamma(N_{\sigma}).$$
The Riemann-Roch formula gives
$$\chi(N_{\sigma})=h_{\omega^{-1}_{\pi}}(\sigma)+2(1-\sg(B)),$$
which implies
$$\dim_{\sigma}\Sect(\cX/B,h)\le \dim T_{\sigma}\Sect(\cX/B,h) =
h^0(N_{\sigma})\ge h + 2(1-\sg(B)),$$
with equality when $N_{\sigma}$ has no higher cohomology.  
It is possible for $h_{\omega^{-1}_{\pi}}(\sigma)<0$, but such sections are
typically confined to subvarieties of $\cX$ (see Remark~\ref{rema:bound}).
In characteristic zero, sections
with deformations dominating $\cX$ have normal bundles that are globally
generated at the generic point, and thus have positive degree.

The discriminant construction gives an alternate approach.
When $k$ is algebraically closed or finite, we
may interpret $\cF\simeq \bP(V)$ for a 
rank-two vector bundle $V \ra C$.  We are using the
fact that the Brauer group of a projective curve over a finite field 
is trivial, essentially by class field theory.   
Let $\cO_{\bP(V)}(1)$ denote the resulting polarization.
Let $\Sect(\cF/C)$ denote the space of sections $\tau:C \ra \cF\simeq \bP(V)$,
again regarded as an open subscheme of the Hilbert scheme of $\cF$.
We have a morphism
$$\begin{array}{rcl}
\alpha: \Sect(\cF/C) & \ra & \Pic(C) \\
\tau & \mapsto & \tau^*\cO_{\bP(V)}(1)
\end{array}
$$
with fibers corresponding to extensions
$$0 \ra N \ra V^{\vee} \ra \tau^*\cO_{\bP(V)}(1) \ra 0.$$
These yield elements of
$$\Hom(N,V^{\vee})=N^{\vee} \otimes V^{\vee}=N^{\vee} \otimes \bigwedge^2 V^{\vee} \otimes V=
\tau^*\cO_{\bP(V)}(1) \otimes V.$$
Given $L \in \Pic(C)$, the sections with $\tau^*\cO_{\bP(V)}(1)=L$
lie in the projectivization 
$\bP(\Gamma(V\otimes L)).$
Thus for $d:=\deg(L)$ sufficiently large, the sections form a Zariski-open dense
subset of a projective
bundle over $\Pic^d(C)$.  The boundary points correspond to unions of sections
with fibers of $\bP(V) \ra C$, reflecting homomorphisms $V^{\vee} \ra \tau^*\cO_{\bP(V)}(1)$ with 
non-vanishing cokernel.
The Riemann-Roch formula implies
$$\chi(V\otimes L) = 2d+\deg(V)+2(1-\sg(C))=2d+\deg(V)-2\sg(C)+2.$$

We summarize this as follows:
\begin{prop}
Retain the notation introduced above, including the choice of a vector bundle $V$
such that $\cF\simeq \bP(V)$.  For each $h\in \bZ$, there exists a $d\in \bZ$
and a morphism
$$\gamma_h:\Sect(\cX/B,h) \ra \Pic^d(C).$$
For $h\gg 0$ this is the composition of an open 
immersion with a projective bundle of relative dimension
\begin{equation} \label{eqn:RR}
2d+\deg(V)-4\sg(B)-\Delta+3.
\end{equation}
\end{prop}

The morphism $\gamma_h$ and the integer $d$ are not canonical, but depend on the choice of $V$.  
Nevertheless, comparing the expected dimensions for $\Sect(\cX/B)$ and $\Sect(\cF/C)$ we find 
$$h + 2(1-\sg(B))=2d+\deg(V) -4\sg(B)-\Delta +3 +\dim \Pic^d(C),$$
which yields the relation
\begin{equation} \label{eqn:hV}
h=2d+\deg(V)-\frac{\Delta}{2}.
\end{equation}

\subsection*{A useful congruence}
Assume $k$ is algebraically closed.  Recall
the set-up in the proof of Proposition~\ref{prop:quadricformulaB}:
We have the base-changed family
$\cX\times_B C \ra C$, the singularities $x_1,\ldots,x_{\Delta} \in \cX \times_B C$,
and the modification:
$$\begin{array}{ccc}
\tilde{\cY} & \ra & \cY \\
\downarrow                     &     &     \\
\cX \times_B C           &  	&
\end{array}
$$

Consider the elementary transformation of $g^*\pi_*H$ associated
with the ordinary
double points \cite{Maruyama}
\begin{equation} \label{eqn:ET}
0 \ra W  \ra g^*\pi_*H \ra Q \ra 0,
\end{equation}
where $Q$ is a skyscraper sheaf supported at $\{x_1,\ldots,x_{\Delta}\}$ 
with length one at each point. 
We can compute
$$\begin{array}{rcl}
\deg(W)&=&\deg(g^*\pi_*H)-\deg(Q)=-2\deg(E)-\Delta\\
&=&
\begin{cases} 0 & \text{ if } \epsilon(\pi)\equiv 0\pmod{2} \\
 -4 & \text{ if } \epsilon(\pi)\equiv 1\pmod{2}.
\end{cases}
\end{array}
$$

The geometric interpretation of the elementary transformation gives
an embedding
$$\cY \hookrightarrow \bP(W^{\vee}).$$
The Fano variety of lines is a disjoint union
$$\cF(\cY/C)=\cF \sqcup \iota^*\cF=\bP(V) \sqcup \bP(\iota^*V).$$
Indeed, $\cF \times_B C$ is non-normal over the discriminant, reflecting
the fact that the two rulings of a smooth quadric surface both specialize to
the rulings of the quadric cone (see \cite[\S 3]{HVAV} for further details).  
In particular, $\cY \ra C$ is the Segre embedding of a product of two copies of $\bP^1$,
isomorphic to $\bP(V)$ and $\bP(\iota^*V)$.  
Rescaling $V$ by tensoring with a suitable line bundle $L$, we can express
$$(V\otimes L) \otimes \iota^*(V\otimes L) \simeq W^{\vee},$$
i.e., $\deg(V\otimes L)=0$ or $1$ depending on the parity of $\epsilon(\pi)$:
\begin{equation} \label{eqn:parity}
\deg(V\otimes L)\equiv \epsilon(\pi)\pmod{2}.
\end{equation}
Eliminating $\epsilon(\pi)$ from the expressions for $\deg(E)$ and $\deg(V)$,
we find 
\begin{equation} \label{eqn:quadricformula}
4\deg(V)\equiv \Delta - 2 \deg(\pi_*H)=\Delta+2\deg(E) \pmod{8}
\end{equation}
This is true regardless of how we normalize $H$ or $V$.

\begin{rema}
The key here is the coincidence of Lie theory
$$\mathfrak{so}(4,\bC)=\mathfrak{sl}(2,\bC) \oplus \mathfrak{sl}(2,\bC),$$
reflecting the equivalence of Dynkin diagrams
$$D_2=A_1 \cup A_1.$$

Bichsel and Knus \cite{BK} compute Clifford algebras
for rank-four quadratic forms taking values in invertible sheaves.  
This gives an alternate approach to the varieties of
maximal isotropic subspaces of $\pi:\cX \ra B$.  Knus, Parimala,
and Sridharan \cite{KPS} develop the dictionary discussed
here using the language of quadratic forms over an affine base.
Auel \cite[\S 5.3]{Auel} and Auel-Bernardara-Bolognesi \cite[Thm.~2.24]{ABB}
address this over more general base schemes.
\end{rema}

\section{Census of quadric fibrations over $\bP^1$}
\label{sect:casebycase}

Our approach here has connections to the work of Ramanan and 
Bhosle on vector bundles over hyperelliptic curves
\cite{DesRam,Bhosle84,Bhosle98,Bhosle2002,Bhosle2010}.
It would be very interesting to work out a complete dictionary
between their work and our approach, with particular attention
to degenerations of the hyperelliptic curves.

We assume $B\simeq \bP^1$ and the discriminant curve $C$ has genus $\sg$.
Equation~\ref{eqn:quadricformula} is equivalent to
\begin{equation} \label{eqn:formula}
\deg(\pi_*H)\equiv \sg+1-2\deg(V) \pmod{4}.
\end{equation}
The expected dimension of the space
of sections over a fixed $L \in \Pic(C)$ (Equation~\ref{eqn:RR}) can be written
\begin{equation} \label{eqn:RRzero}
\chi(V\otimes L)-1=\deg(V\otimes L)-2\sg+1.
\end{equation}
The normalized bundle $E\sim (\pi_*H)^{\vee}$ satisfies
$$\deg(E)=\begin{cases} -\sg-1 & \text{ if $\epsilon(\pi)\equiv 0\pmod{2}$} \\
                        -\sg+1 & \text { if $\epsilon(\pi)\equiv 1\pmod{2}$.}
	\end{cases}
$$
We can decompose
$$
\begin{array}{c}\pi_*H=\cO_{\bP^1}(-a_1) \oplus
\cO_{\bP^1}(-a_2) \oplus
\cO_{\bP^1}(-a_3) \oplus
\cO_{\bP^1}(-a_4), \\ 
a_1 \le a_2 \le a_3 \le a_4;
\end{array}
$$
for simplicity, from now on assume this is as `balanced' as possible, i.e., $a_4-a_1 \le 1$.  
We refer the reader to Section 1 of \cite{LPS} for a more thorough
classification.  
\begin{rema}
Any vector bundle $E_0$ on $\bP^1$ admits a small deformation to a balanced 
vector bundle $E$, i.e., 
$E\simeq \oplus_{j=1}^r \cO_{\bP^1}(-m_j)$ where $|m_i-m_j|\le 1$ for each $i,j=1,\ldots,r$.
Indeed, the splitting $E_0\simeq \oplus_{j=1}^r\cO_{\bP^1}(-n_j)$ 
can be deformed to a non-trivial
extension that is balanced \cite{Shatz}.  If $\Sym^2(E_0^{\vee})\otimes I$ is globally generated (here $I$ 
is invertible of degree $0$ or $1$, as in Section~\ref{sect:QFDH}) then 
$H^1(\Sym^2(E_0^{\vee})\otimes I)=0$.
Cohomology-and-base-change implies that sections in 
$\Gamma(\bP^1,\Sym^2(E_0^{\vee})\otimes I)$ arise as specializations of 
sections in $\Gamma(\bP^1,\Sym^2(E^{\vee})\otimes I)$.  Thus the balanced
bundles we consider are generic for large classes of quadric surface fibrations.
\end{rema} 

\subsection*{Case 1} Here we have $a_1=a_2=a_3=a_4$ 
so that 
$$\bP((\pi_*H)^{\vee})\simeq \bP(\cO_{\bP^1}^{\oplus 4}))\simeq \bP^1 \times \bP^3.$$
The equation of $\cX$ is a form of bidegree $(n,2)$,
thus $h(\cX)=16n$, $\Delta=4n$, and
$\sg=2n-1$.  
The normalized rank-four bundle is
$$E=\begin{cases}  \cO_{\bP^1}(-m)^{\oplus 4} & \text{ if $n=2m+1$ odd} \\
  \cO_{\bP^1}(-m)^{\oplus 4} & \text{ if $n=2m$ even.} 
	   \end{cases}
$$ 
Equation~\ref{eqn:formula} implies that
$\deg(V) \equiv n \pmod{2}$.
In light of (\ref{eqn:RRzero}), we take 
$$\deg(V)=\begin{cases}   4n-3=2\sg-1 & \text{ if $n$ odd} \\
			  4n-2=2\sg   & \text{ if $n$ even.}
	   \end{cases}
$$ 

A form of bidegree $(n,2)$ depends on $(n+1)\cdot 10-1=10n+9$ parameters;
taking into account the automorphisms of $\bP^1 \times \bP^3$,
we are left with $10n-9=5\sg-4$ free parameters.

\subsection*{Case 2} Here we have $a_1+1=a_2=a_3=a_4$ so that
$$\bP((\pi_*H)^{\vee})\simeq \bP(\cO_{\bP^1}(-1) \oplus \cO_{\bP^1}^{\oplus 3}))
\hookrightarrow \bP^1 \times \bP^4.$$
The $\bP^3$-bundle is given by a form of degree $(1,1)$,
and the second equation defining $\cX$ has bidegree $(n,2)$.  
We have $\omega_{\cX/\bP^1}=\cO_{\cX}(n+1,-2)$ and $h(\cX)=16n+8$;
then $\Delta=4n+2$ and $\sg=2n$.  
The normalized twist of $(\pi_*H)^{\vee}$ is:
$$E=\begin{cases}  \cO_{\bP^1}(-m-1)\oplus \cO_{\bP^1}(-m)^{\oplus 3} & \text{ if $n=2m+1$ odd} \\
  \cO_{\bP^1}(-m-1)\oplus \cO_{\bP^1}(-m)^{\oplus 3} & \text{ if $n=2m$ even.} 
	   \end{cases}
$$ 
Formula (\ref{eqn:formula}) implies
$\deg(V) \equiv n \pmod{2}$.
In light of (\ref{eqn:RRzero}), we take 
$$\deg(V)=\begin{cases}   4n-1=2\sg-1 & \text{ if $n$ odd} \\
			  4n=2\sg   & \text{ if $n$ even.}
	   \end{cases}
$$ 

Note that $\cX$ is cut out by a form of bidegree $(1,1)$ and a 
form of bidegree $(n,2)$, determined modulo multiples of the 
first form.  The former depends on nine parameters, the latter
on $10n+14$ parameters (even when $n=0$).  Taking automorphisms
of $\bP^1 \times \bP^4$ into account, we are left with a total of
$$10n-4=5\sg-4$$
free parameters.  When $n=0$, this should be understood to mean that
the families admit positive-dimensional automorphism groups.  

\subsection*{Case 3} Here we have $a_1+1=a_2+1=a_3=a_4$ hence
$$\bP((\pi_*H)^{\vee})\simeq \bP(\cO_{\bP^1}(-1)^{\oplus 2} \oplus \cO_{\bP^1}^{\oplus 2}))
\hookrightarrow \bP^1 \times \bP^5.$$
The $\bP^3$-bundle is given by two forms of degree $(1,1)$, with $\cX$ defined by one
additional equation of bidegree $(n,2)$.  The dualizing sheaf 
$\omega_{\cX/\bP^1}=\cO_{\cX}(n+2,-2)$, $\Delta=4n+4$, $\sg=2n+1$, and 
$$h(\cX)=16n+16.$$
Here we have
$$E=\begin{cases}  \cO_{\bP^1}(-m-1)^{\oplus 2} \oplus \cO_{\bP^1}(-m)^{\oplus 2} & \text{ if $n=2m+1$ odd} \\
  \cO_{\bP^1}(-m-1)^{\oplus 2}\oplus \cO_{\bP^1}(-m)^{\oplus 2} & \text{ if $n=2m$ even.} 
	   \end{cases}
$$ 
Formula (\ref{eqn:formula})
implies 
$\deg(V) \equiv n \pmod{2}$.
To get smallest possible non-negative dimensions in (\ref{eqn:RRzero}), we take 
$$\deg(V)=\begin{cases}   4n+1=2\sg-1 & \text{ if $n$ odd} \\
			  4n+2=2\sg   & \text{ if $n$ even.}
	   \end{cases}
$$ 

We compute the number of free parameters:  The forms of bidegree $(1,1)$
correspond to a point of $\Gr(2,\Gamma(\cO_{\bP^1 \times \bP^5}(1,1)))$,
which has dimension $20$.  When $n>0$, the form of bidegree $(n,2)$ modulo the first
two forms depends on $(n+1)21-2(n)6+(n-1)-1=10n+19$ parameters.  Taking
automorphisms into account, we obtain
$$20+(10n+19)-38=10n+1=5(2n+1)-4=5\sg-4$$
parameters.  When $n=0$ and $\sg=1$, the construction depends on two parameters.  

\subsection*{Case 4} In this case $a_1=a_2=a_3=a_4-1$ hence
$$\bP((\pi_*H)^{\vee})\simeq \bP(\cO^{\oplus 3}_{\bP^1}(-1) \oplus \cO_{\bP^1}))
\hookrightarrow \bP^1 \times \bP^6.$$
The $\bP^3$-bundle is given by three forms of degree $(1,1)$, with $\cX$ defined by one
additional equation of bidegree $(n,2)$.  The dualizing sheaf 
$\omega_{\cX/\bP^1}=\cO_{\cX}(n+3,-2)$, $\Delta=4n+6$, $\sg=2n+2$, and 
$$h(\cX)=16n+24.$$
In this case
$$E=\begin{cases}  \cO_{\bP^1}(-m-1)^{\oplus 3} \oplus \cO_{\bP^1}(-m) & \text{ if $n=2m+1$ odd} \\
  \cO_{\bP^1}(-m-1)^{\oplus 3}\oplus \cO_{\bP^1}(-m) & \text{ if $n=2m$ even.} 
	   \end{cases}
$$ 

Formula (\ref{eqn:formula}) implies
$\deg(V) \equiv n \pmod{2}$.  
In light of (\ref{eqn:RRzero}), we take 
$$\deg(V)=\begin{cases}   4n+3=2\sg-1 & \text{ if $n$ odd} \\
			  4n+4=2\sg   & \text{ if $n$ even.}
	   \end{cases}
$$

We compute free parameters:  The forms of bidegree $(1,1)$
correspond to a point of $\Gr(3,\Gamma(\cO_{\bP^1 \times \bP^6}(1,1)))$,
which has dimension $33$.  The form of bidegree $(n,2)$, modulo the first
three forms, depends on 
$$
(n+1)28-3(n)7+3(n-1)-1=10n+24
$$ 
parameters.  Taking
automorphisms into account, we obtain
$$33+(10n+24)-51=10n+6=5(2n+2)-4=5\sg-4$$
parameters.  

\begin{rema} \label{rema:genuszero}
In our analysis the case where the discriminant curve has
genus zero stands out;  we have yet to exhibit an
example where $\sg=0$ and $\deg(V)$ is odd.  This
may be interpreted as the $n=-1$ instance of Case 4
above.  

Specifically, there are quadric surface fibrations 
$$\cX \subset \bP((\pi_*H)^{\vee})\simeq \bP(\cO^{\oplus 3}_{\bP^1}(-1) \oplus \cO_{\bP^1})$$
that do not arise as restrictions of hypersurfaces in $\bP^1 \times \bP^6$.  
These correspond to global sections of
$$\Sym^2(\pi_*H) \otimes \cO_{\bP^1}(-1)= \cO_{\bP^1}(1)^{\oplus 6} \oplus
\cO_{\bP^1}^{\oplus 3} \oplus \cO_{\bP^1}(-1),$$
which necessarily contain the distinguished section $\sigma:\bP^1 \ra \bP((\pi_*H)^{\vee})$.  
Projecting from $\sigma$, we obtain
$$\begin{array}{rcccl}
 	&	&   \Bl_{\sigma(\bP^1)}(\cX)  &		&  	\\
        & \swarrow &			&\stackrel{\beta}{\searrow}    & 	\\
 \cX    &	   &			&	    & \bP^2 \times \bP^1 \\
        & \searrow &			& \swarrow    & 	\\
	&	   &	\bP^1		&	  	&
\end{array}
$$
where $\beta$ blows up a genus-zero bisection $\cZ \subset \bP^2 \times \bP^1 \ra \bP^1$.  
The bisection $\cZ$ is a complete intersection of hypersurfaces of bidegrees $(1,0)$
and $(2,1)$ in $\bP^2 \times \bP^1$.  The former takes the form $\ell \times \bP^1$,
where $\ell \subset \bP^2$ is a line, 
and coincides with the proper transform of the 
exceptional divisor of $\Bl_{\sigma(\bP^1)}(\cX) \ra \cX$.
Constant sections of $\bP^2 \times \bP^1 \ra \bP^1$ induce sections of $\pi:\cX \ra \bP^1$;
points of $\ell$ give rise to reducible curves, consisting of
the union of $\sigma(\bP^1)$ and a line in a fiber of $\pi$ incident to $\sigma(\bP^1)$.

The families constructed here admit positive-dimensional automorphism groups.
\end{rema}

We summarize our computations in the following table:

\

\begin{tabular}{|r|c|c|c|}
\hline
Case	& $\Delta\pmod{8}$   & $n\equiv\deg(V)\pmod{2}$  & $\sg\pmod{4}$ \\
\hline
  1     &        0		 &      0		   &   -1 \\
	& 	 4		 &      1		   &	1 \\
  2     & 	 2		 &	0		   &    0 \\
	&	 6		 &	1		   &    2 \\
  3     &        4		 &      0		   &    1 \\
        &        0		 &      1 		   &   -1 \\
  4     &        6               &	0		   &    2 \\
 	&        2		 &      1		   &    0 \\
\hline
\end{tabular}

\subsection*{Parameter counts and relations to moduli spaces of bundles}
When $V$ is a rank-two vector bundle over $C$, we have 
$$\chi(\End(V))=4(1-\sg).$$
When $V$ is simple, the moduli space has dimension $4\sg-3$;  fixing the
determinant gives a moduli space of dimension $3\sg-3$.  
Taking into account the fact that $\bP(V\otimes L)\simeq \bP(V)$
for each $L \in \Pic(C)$, 
the corresponding moduli space of $\bP^1$-bundles over $C$
also depends on $3\sg-3$ parameters.  
Hyperelliptic curves depend on $2\sg-1$ parameters so the
total number of free parameters is
$$3\sg-3+2\sg-1=5\sg-4,$$
the number of free parameters we observed in each case.

\section{Hecke correspondences and elementary transformations}  
\label{sect:hecke}
The data tabulated above exhibit an involution preserving 
$\Delta\pmod{8}$ and $\sg\pmod{4}$ but reversing the parity of 
$\deg(V)$ and altering $\deg(\pi_*H)\pmod{4}$ by two.  This can be explained
geometrically 
via elementary transformations. 

Fix a smooth fiber $\cX_p$ of $\pi$ and
a line $\ell \subset \cX_p$.  Applying an elementary transformation along
$\ell$ converts $\bP(\cO_{\bP^1}^{\oplus 4})$ to $\bP(\cO_{\bP^1}(-1)^{\oplus 2}
\oplus \cO_{\bP^1}^{\oplus 2})$ (resp.
$\bP(\cO_{\bP^1}(-1) \oplus \cO_{\bP^1}^{\oplus 3})$ to $\bP(\cO_{\bP^1}(-1)^{\oplus 3}
\oplus \cO_{\bP^1})$).  The proper transform $\tilde{\pi}:\tilde{\cX}\ra \bP^1$
of $\cX$ is still a quadric surface fibration
with the same degenerate fibers.   
This also induces an elementary transformation of $\cF(X)=\bP(V)\ra C$ 
at the point $\ell$, which is $\cF({\tilde X})$;  this changes the parity
of the degree of this rank-two bundle.  

This process does change the heights of sections of $\pi:\cX \ra \bP^1$.
Suppose that $\sigma:\bP^1 \ra \cX$ is a section disjoint from $\ell$,
with proper transform $\tilde{\sigma}:\bP^1 \ra \tilde{\cX}$.  The
birational map $\cX \dashrightarrow \tilde{\cX}$ factors
$$\begin{array}{rcccl}
 	&	& \Bl_{\ell}(\cX)    &    &  \\
        &\swarrow &           & \searrow &   \\
  \cX   &        &            &         & \tilde{\cX},
\end{array}
$$
where the right arrow blows down the proper transform of $\ell$.  Thus we find
$$\deg({\tilde \sigma}^*\omega_{\tilde{\cX}/\bP^1})=
\deg(\sigma^*\omega_{\cX/\bP^1})-1$$
and 
$$h_{\omega^{-1}_{\tilde{\pi}}}(\tilde{\sigma})=
h_{\omega^{-1}_{\pi}}(\sigma)+1.
$$
Thus taking elementary transformations along lines incident to a section
reduces the height of that section.  

If we apply {\em two} elementary transformations to $\cX \ra \bP^1$, the
resulting quadric fibration has the same numerical invariants and 
discriminant curve $C\ra B$.  And of course, the resulting fibrations
are birational over $\bP^1$ but not isomorphic, as the corresponding vector
bundles are related by an elementary transformation.  This is an instance
of a {\em Hecke correspondence} on the moduli space of vector bundles over $C$;
these have been studied by many authors, e.g., \cite{NS}.

\section{Stable bundles and cohomology}
\label{sect:SBC}
In Section~\ref{sect:reduction}, we saw how to translate the existence of sections
of quadric surface bundles $\cX\ra B$ to the existence of sections
of a ruled surface $\cF\simeq \bP(V) \ra C$, where $C$ is the discriminant double cover of $B$
and $V$ is a rank-two vector bundle over $C$.  
The behavior of the sections of a ruled surface depends on the characteristics
of this vector bundle;  in general, there is little we can say uniformly without making
some assumptions on the vector bundle.  

Throughout this section, $C$ is a smooth projective curve of genus $\sg$ over an
algebraically closed field.  
Recall that a locally-free sheaf $V$ over $C$ is {\em stable} (resp.~{\em semi\-stable})
if, for every locally-free quotient
$$V \ra W \ra 0$$
we have 
$$\deg(W)/\rank(W) > (\text{ resp.} \ge)\, \deg(V)/\rank(V).$$
A vector bundle is {\em stable} if the associated locally-free sheaf of sections is stable.  
The stability of a vector bundle is not affected by tensoring it by a line bundle, or by
taking its dual.

\subsection*{General facts on rank-two bundles}

Here we collect more refined vanishing results for 
stable bundles of rank two, which will be useful for effective estimates 
for the numbers of sections with prescribed properties.  

Let $V$ be a rank-two vector bundle over a $C$ and $M\subset V$
an invertible subsheaf of maximal degree.  Recall that $V/M$
is invertible and
$$\Gamma(V \otimes M_1^{-1})=0$$
for all invertible $M_1$ with $\deg(M_1)>\deg(M)$ \cite[V.2]{Hart}.
This implies that $\dim \Gamma(V\otimes M^{-1}) \le 2$.  

\begin{prop}  \label{prop:bounds}
Let $V$ be a 
vector bundle on $C$ of rank two.  
\begin{itemize}
\item
An invertible subsheaf of maximal degree $M\subset V$ satisfies
\cite{Nagata}
$$\frac{\deg(V)-\sg}{2} \le \deg(M).$$
If $V$ is semistable then $\deg(M)\le \deg(V)/2$.  
\item
If $V$ is semistable then for a generic line bundle
$L$ on $C$ of degree zero we have
\cite[Prop. 1.6.2]{Raynaud}
$$h^0(V\otimes L)=\mathrm{max}(0,\chi(V)).$$
\end{itemize}
\end{prop}

\begin{rema} 
The original formulation of Nagata's theorem referenced above is worth mentioning:
Let $\bP(V) \ra C$, where $V$ is a rank-two
vector bundle over $C$  (not necessarily semistable).  Then there 
exists a section $\tau:C\ra \bP(V)$ such that
$$\tau(C)\cdot \tau(C) \le \sg.$$
\end{rema}

\begin{coro} \label{coro:gg}
Let $V$ be a semistable vector bundle of rank two over $C$.  
If $\deg(V) \ge 3\sg+2$ then $V\otimes L$ is globally generated with vanishing
higher cohomology for generic $L\in \Pic^0(C)$.  
When $\deg(V) \ge 4\sg-1$ then $V$ itself is globally generated with vanishing
higher cohomology.  
\end{coro}
\begin{proof}
For the first assertion,
Proposition~\ref{prop:bounds} implies we may express $V$ as an extension of invertible
sheaves
$$0 \ra M \ra V \ra V/M \ra 0,$$
where $\deg(M)\ge \sg+1$ and $\deg(V/M)\ge \frac{3}{2}\sg+1$ 
if $\deg(V)\ge 3\sg +2$.  
For generic $L$, $M\otimes L$ and $(V/M) \otimes L$ are globally generated with
vanishing higher cohomology, so the same is true for $V\otimes L$.  

We are grateful to N.~Hoffmann for suggesting improvements on the bound
for the second assertion.  
To prove that $V$ is globally generated, it suffices to show that for each
expression
$$0 \ra U \ra V \ra Q \ra 0$$
with $Q$ a torsion sheaf of length one, we have $\Gamma(U) \subsetneq \Gamma(V)$.  
If we can show that $H^1(U)=0$, it will follow that $H^1(V)=0$ and 
$$\dim \Gamma(U)=\chi(U)=\chi(V)-1<\dim \Gamma(V).$$
However, if $H^1(U)$ were non-vanishing then $\Gamma(\omega_C \otimes U^{\vee})\neq 0$ by Serre
duality.  Let $N \subset \omega_C \otimes U^{\vee}$ denote the saturation of some non-vanishing
section, whence $\deg(N)\ge 0$;  consider the resulting extension
$$0 \ra N \ra \omega_C\otimes U^{\vee} \ra  \omega^2_C\otimes N^{-1} \otimes \det(U)^{-1} \ra 0.$$
On dualizing, we obtain
$$0 \ra N \otimes \det(U) \otimes \omega_C^{-1} \ra U \ra N^{-1}\otimes \omega_C \ra 0.$$
We may regard the first term as a subsheaf of $V$, so semistability implies
$$\deg(N \otimes \det(U) \otimes \omega_C^{-1}) \leq \deg(V)/2$$
and 
$$\deg(N)+\deg(V)-1-(2\sg-2)  \leq \deg(V)/2.$$
Thus we conclude
$$\deg(N) \leq \frac{-\deg(V)}{2}+2\sg-1 < 0,$$
a contradiction.
\end{proof}

\section{Projective bundles over limits of hyperelliptic curves}
\label{sect:projbundle}

The purpose of this and the subsequent section is to analyze 
how sections of quadric surface fibrations specialize as the 
base of the fibration degenerates to a nodal curve.  Essentially,
the excellent {\em a priori} control we have for sections of quadric
surface fibrations gives structure to how sections `break' as the
fibration breaks into a union of two fibrations of smaller height.  
We carry out this
analysis with a view toward understanding the behavior of sections
of del Pezzo fibrations of smaller degree over $\bP^1$. 

\subsection*{Sections of projective bundles over nodal curves}
Let $C$ be a nodal projective curve of arithmetic genus $\sg$
over an algebraically closed field.
\begin{lemm} \label{lemm:projbund}
If $P\ra C$ is a projective bundle then there exists a vector bundle
$V$ over $C$ such that $P=\bP(V)$.  Sections $t:C \ra P$ 
with $t^*\cO_{\bP(V)}(1)\simeq L$ correspond to 
short exact sequences
$$0 \ra N \ra V^{\vee} \ra L \ra 0,$$
or equivalently, elements of 
$$\bP(\Gamma(C,V\otimes L))$$
such that the induced $V^{\vee} \ra L$ is surjective.  
\end{lemm}
\begin{proof}
The obstruction to lifting a cocycle in $\PGL_{r}$
to $\GL_r$ lies in the Brauer group, which is trivial on a 
curve.  This gives the first assertion.
The second assertion is the standard characterization of
morphisms into projective space.   
\end{proof}

Let $e$ denote an integer-valued function from the set of irreducible
components of $C$ and $|e|$ the sum of this function over these components.
Consider the irreducible component of the Hilbert scheme
$\Sect(P/C,e)$ containing the sections
$$\{\tau:C\ra P: \deg(\tau^*\cO_{\bP(V)}(1))=e \}.$$
We have a rational map
$$
\begin{array}{rcl}
\alpha_e: \Sect(P/C,e) & \dashrightarrow & \Pic^e(C) \\
      \tau      & \mapsto & \tau^*(\cO_{\bP(V)}(1)).
\end{array}
$$
We are interested in those $e$ such that $\alpha_e$ is dominant.
By Lemma~\ref{lemm:projbund},
these include all $e$ such that, for generic $L \in \Pic^e(C)$,
we have
\begin{equation} \label{eqn:condition}
\Gamma(V\otimes I \otimes L)\subsetneq \Gamma(V\otimes L)
\end{equation}
for each ideal sheaf $I \subsetneq \cO_C$.  
Indeed, if $V^{\vee} \ra L$ fails to be surjective then its image
is isomorphic to $L\otimes I$ for some ideal sheaf $I$.

\subsection*{Applications to degenerate quadric fibrations}

Here, a {\em degenerate quadric surface fibration} consists of
\begin{itemize}
\item{a connected nodal curve 
$$B:=B_1\cup_p B_2$$
with a single node $p$;}
\item{a flat morphism from a projective scheme 
$$\pi:\cX \ra B,$$ 
such that the restrictions
$$\pi_j:=\pi|\cX_j=\cX\times_B B_j \ra B_j, \quad  j=1,2,$$
are quadric surface fibrations smooth over $p$ with square-free 
discriminant elsewhere.}
\end{itemize}
Let $g:C\ra B$ denote the discriminant curve;  note that
$$C=C_1 \cup_{q,r}C_2, \quad g(q,r)=p$$
where $g|C_i: C_i \ra B_i$ is a double cover.  
The Fano variety $\cF$ of lines on $\cX$ remains a $\bP^1$-bundle
over $C$.  We can express $\cF=\bP(V)$, where $V$ is a rank-two
vector bundle on $C$ by Lemma~\ref{lemm:projbund}.  The argument
of Section~\ref{sect:reduction} still yields a natural identification between
$\Sect(\cX/B)$ and $\Sect(\cF/C)$.
As before, we define
$$\epsilon(\pi)\equiv \deg(V) \pmod{2},$$
so that $\epsilon(\pi)\equiv \epsilon(\pi_1)+\epsilon(\pi_2)$.  
This is invariant under deformations of $\pi$, including smoothings
to quadric surface fibrations over smooth curves.

\begin{prop} \label{prop:new}
Let $\pi:\cX \ra B$ be a degenerate quadric surface fibration as described above,
over a curve of genus zero.  Assume that the discriminant curve $C=C_1 \cup C_2 \ra B$
has genus $\sg$ and admits a component $C_1 \simeq \bP^1$.  Set 
$$h=\begin{cases}   \sg - 1  & \text{if }\epsilon(\pi) \equiv 0\pmod{2} \\
		   \sg - 2  & \text{if }\epsilon(\pi) \equiv 1\pmod{2}
	\end{cases}
$$
and consider
$$\gamma_h:\Sect(\cX/B,h) \ra \Pic(C).$$
\begin{enumerate}
\item{
If $\epsilon(\pi_1)\equiv \epsilon(\pi_2)\equiv 0$ then $\gamma_h$ dominates two components of the Picard
variety, over which the generic fiber $\simeq \bP^1$.}
\item{
If $\epsilon(\pi_1)\equiv \epsilon(\pi_2)\equiv 1$ then $\gamma_h$ dominates three components of the Picard
variety, over which the generic fiber $\simeq \bP^1$.}
\item{
If $\epsilon(\pi_1)\not \equiv \epsilon(\pi_2)$ then $\gamma_h$ dominates two components of the Picard
variety, over which it is birational.}
\end{enumerate}
\end{prop}
\begin{proof}
In light of the analysis in Section~\ref{sect:casebycase}, we normalize
$$\deg(V)=\begin{cases} 2\sg & \epsilon(\pi)\equiv 0\pmod{2} \\
		        2\sg-1 &\epsilon(\pi)\equiv 1\pmod{2}
	\end{cases}
$$
which means that $\chi(V)=2$ in the even case and $1$ in the odd case.
Note that a generic vector bundle of this degree on a 
smooth projective curve of genus $\sg$ has no higher cohomology
(see Proposition~\ref{prop:bounds}).  

\subsection*{Odd case}
Our first subcase is
$$\deg(V|C_1)\equiv 0 \pmod{2}, \quad 
\deg(V|C_2)\equiv 1 \pmod{2}.$$
The possibilities compatible with (\ref{eqn:condition}) are:
\begin{itemize}
\item{$\deg(V\otimes L|C_1)=2, \quad \deg(V\otimes L|C_2)=2\sg-3$}
\item{$\deg(V\otimes L|C_1)=0, \quad \deg(V\otimes L|C_2)=2\sg-1$}
\end{itemize}
In the first case, $V\otimes L|C_1 \simeq \cO_{\bP^1}(1)^{\oplus 2}$
and $\Gamma(V\otimes L|C_2)$ admits a unique non-zero section, up to
scalar.  On gluing we see that $V\otimes L$ admits a unique section
as well.  Otherwise, $V\otimes L|C_1 \simeq \cO_{\bP^1}^{\oplus 2}$ 
which is globally generated by two sections.  Since $V\otimes L|C_2$
has a three-dimensional space of sections, after gluing we have a unique
section up to scalar.  

The other subcase is
$$\deg(V|C_1)\equiv 1 \pmod{2}, \quad 
\deg(V|C_2)\equiv 0 \pmod{2},$$
which leads to the possibilities:
\begin{itemize}
\item{$\deg(V\otimes L|C_1)=1, \quad \deg(V\otimes L|C_2)=2\sg-2$}
\item{$\deg(V\otimes L|C_1)=-1, \quad \deg(V\otimes L|C_2)=2\sg$}
\end{itemize}
In the first instance, $V\otimes L|C_1 \simeq \cO_{\bP^1} \oplus \cO_{\bP^1}(1)$
which admits three sections,
and $V\otimes L|C_2$ generally has a two-dimensional space of sections.  
In the second instance, $V\otimes L|C_1 \simeq \cO_{\bP^1} \oplus \cO_{\bP^1}(-1)$
which admits a single section and $V\otimes L|C_2$ generally has a four-dimensional
space of sections.  
In both instances, we have a unique section up to scalar.  

\subsection*{Even case}
Our first subcase is
$$\deg(V|C_1)\equiv \deg(V|C_2) \equiv 0\pmod{2}.$$
The possibilities consistent with (\ref{eqn:condition}) are limited to:
\begin{itemize}
\item{$\deg(V\otimes L|C_1)=2, \quad \deg(V\otimes L|C_2)=2(\sg-1)$}
\item{$\deg(V\otimes L|C_1)=0, \quad \deg(V\otimes L|C_2)=2\sg$.}
\end{itemize}
In the former case, we have $V\otimes L|C_1 \simeq \cO_{\bP^1}(1)^{\oplus 2}$,
which is globally generated with four sections.
For generic $L$, we find that $V\otimes L|C_2$ 
admits a two-dimensional space of sections.
Overall, we find that $V\otimes L$ has a two-dimensional
space of sections.  
In the latter case, we have $V\otimes L|C_1 \simeq \cO_{\bP^1}^{\oplus 2}$
which is globally generated with two sections.  For generic $L$,
$V\otimes L|C_2$ admits a four-dimensional space of sections.  
After gluing we find that $\Gamma(V\otimes L)$ is two-dimensional.  

Our second subcase is
\begin{equation}
\deg(V|C_1)\equiv \deg(V|C_2) \equiv 1\pmod{2}. \label{badcase}
\end{equation}
Condition (\ref{eqn:condition}) allows the following {\em three} possibilities:
\begin{itemize}
\item{$\deg(V\otimes L|C_1)=1, \quad \deg(V\otimes L|C_2)=2\sg -1$}
\item{$\deg(V\otimes L|C_1)=-1, \quad \deg(V\otimes L|C_2)=2\sg + 1$}
\item{$\deg(V\otimes L|C_1)=3, \quad \deg(V\otimes L|C_2)=2\sg - 3$}
\end{itemize}
Again, in each case we find that $\Gamma(V\otimes L)$ is 
two-dimensional.  
\end{proof}

\section{Limits of sections and N\'eron models of intermediate Jacobians}
\label{sect:Neron}

We retain the notation of Section~\ref{sect:projbundle}.  Let
$$\bD=\{t\in \bC: 0<t<1 \}$$ 
denote a complex disc,
$\cX(t)$ a family of quadric surface fibrations
specializing to $\cX=\cX(0)$ over $\bD$,
and $\cC(t)$ the corresponding family of discriminant curves specializing
to $C$.  
Note that this family is {\em not} stable, as the component
$C_1\simeq \bP^1 \subset C$ must be contracted in a stable reduction.  

The intermediate Jacobian $\IJJ(\cX(t))$ is isomorphic
to the Jacobian $\JJ(\cC(t))$ of the discriminant curve $\cC(t)$.  
Here we compute the special
fiber $\tilde{\JJ}_e(0)$ of the N\'eron model
$$\tilde{\JJ}_e \ra \bD$$
of the intermediate Jacobians $\IJJ(\cX(t))$, following
the exposition of \cite[pp.313-314]{GGK}, which draws on previous work of
I. Nakamura \cite{Nakamura}.  

There is a basis for the homology of $\cC(t)$ such that the
monodromy matrix takes the form
$$\left( \begin{matrix}
	I_{2\sg-2}  &  0 \\
 	    0     &  T
\end{matrix} \right), \quad T=\left( \begin{matrix} 1 & 2 \\
						    0 & 1 \end{matrix} \right),$$
where $I_{2\sg-2}$ is the identity matrix of the indicated size.  The logarithm
of this matrix takes the form
$$\left( \begin{matrix}
	0  &  0 \\
 	    0     & N 
\end{matrix} \right), \quad N=\left( \begin{matrix} 0 & 2 \\
						    0 & 0 \end{matrix} \right).$$
Applying formula II.C.1 of \cite{GGK}, we obtain an exact sequence
$$0 \ra \JJ(C) \ra \tilde{\JJ}_e(0) \ra G \ra 0,$$
where $G=\bZ/2\bZ$ is the group of connected components.  
Note that $\JJ(C)$ is itself an extension
$$0 \ra \bG_m \ra \JJ(C) \ra \JJ(C_2) \ra 0.$$

This extension is important because it is the target of the cycle
class map for {\em limits} of one-cycles homologous to zero.
Let $\cX(t)$ be a family of quadric fibrations with $\cX\simeq \cX(0)$ and
$\cX(t)$ non-singular for $t\neq 0$.  Let $Z^2(\cX(t))$ denote the 
codimension-two cycles of $\cX(t)$ homologous to zero, e.g., differences of two
sections of $\pi(t):\cX(t) \ra \bP^1$.  Let $\cZ(t)$ denote a family of
such cycles in $\cX(t)$, with $\cZ(t)$ homologous to zero for $t\neq 0$.  
Note however that $\cZ(0)$ need not be homologous to zero, e.g., when it
is a difference of two sections of 
$$\pi:\cX \ra B=B_1\cup_p B_2 $$ 
whose heights are equal but are allocated differently between the components
of $\cX$.  Nevertheless, the Abel-Jacobi images of the $\cZ(t)$ yield
a section of 
$$\tilde{\JJ}_e \ra \bD$$
and thus an element $\gamma(\cZ(0))\in \tilde{\JJ}_e(0)$.  

This is visibly consistent with the description in Section~\ref{sect:projbundle},
except in the case (\ref{badcase}) where there are {\em three} kinds of 
sections of the projective bundle but only {\em two} components of the N\'eron model.
We explain the geometry of the induced mapping
$$\gamma:\Sect(\cX/B,h) \ra \tilde{\JJ}_e(0),$$
where the height is chosen so that the sections correspond
to elements of $\Gamma(V\otimes L)$ with
$\deg(V\otimes L)=2\sg$.  

\begin{prop}
Recall the notation and assertions of Proposition~\ref{prop:new}.  
For assertions 1 and 3, the two components of $\Pic(C)$ dominated by $\gamma_h$ 
correspond to the two components of $\tilde{\JJ}_e(0)$.  
For assertion 2, sections corresponding to the cases
$$\deg(V\otimes L|C_1)=-1, \quad \deg(V\otimes L|C_2)=2\sg + 1$$
and 
$$\deg(V\otimes L|C_1)=3, \quad \deg(V\otimes L|C_2)=2\sg - 3$$
are mapped onto the same connected component of $\tilde{\JJ}_e(0)$.  
\end{prop}
\begin{proof}
Only the last statement requires proof.  
Indeed, the fact that these are related by an involution can be seen 
by tensoring $L$ in the first case by $\cO(C_2)|C_2$, where
$C_2$ is regarded as a Cartier divisor on the total space of $C_t$.  
This increases the degree of the vector bundle on $C_1$ (and
decreases the degree on $C_2$) by four.  
\end{proof}

The functorial properties of N\'eron-models allow us to compactify
$$\tilde{\JJ}_e(0)\subset \overline{\JJ}_e(0),$$
where the latter is a $\sg$-dimensional toroidal compactification
over the $(\sg-1)$-dimensional abelian variety $\JJ(C_2)$.  Thus its fibers consist of
pairs of $\bP^1$'s meeting in two nodes.  The fibers of
$$\oSect(\cX/B,h) \dashrightarrow \overline{\JJ}_e(0)\ra \JJ(C_2),$$
from the irreducible component of the Hilbert scheme compactifying
$\Sect(\cX/B,h)$, do have three components corresponding to the
three cases of (\ref{badcase}).  

\begin{rema}
The analysis of limiting intermediate Jacobians here is also reminiscent of
Caporaso's compactification of the relative Picard scheme
over the moduli space of stable curves \cite{Caporaso}, especially Section 7.3 
which addresses curves with two components.  
The two cases of (\ref{badcase}) identified to the same component
of $\tilde{\JJ}_e(0)$ correspond to strictly semistable line bundles
that are identified by the Geometric Invariant Theory.  
\end{rema}

\begin{rema}
An alternative approach to limits of intermediate Jacobians via log geometry can be found
in \cite{KNU}.
\end{rema}

\section{Stability and discriminant curves}
\label{sect:SDC}

Let $\pi:\cX \ra B$ be a quadric surface fibration with square-free discriminant.
Let $\cF \ra C \stackrel{g}{\ra} B$ be the Fano variety of lines, realized as a $\bP^1$-bundle
over the discriminant, and $\iota:C\ra C$ the covering involution over $B$.  
Consider a section $\sigma:B\ra \cX$ of $\pi$ and the corresponding
section $\tau:C \ra \cF$ described in Section~\ref{sect:reduction}. 

If $\pi$ is smooth then the 
restriction of scalars
$\varpi:\Res_{C/B}(\cF) \ra B$
is isomorphic to $\cX \ra B$; this yields an isomorphism of normal bundles
$$g_*N_{\tau}=N_{\sigma}.$$
Furthermore, expressing $\cF$ as the projectivization $\bP(V)$ of a vector bundle on $C$,
we can write
$$g^*((\pi_*H)^{\vee})\simeq V \otimes \iota^*V$$
for a suitable normalization of $H$.  This reflects the fact that $\cX \hookrightarrow \bP((\pi_*H)^{\vee})$
is the Segre embedding of $\Res_{C/B}(\cF)$.

We extend these formulas to quadric surface fibrations with square-free discriminant, with a view toward
comparing various 
notions of stability and 
applying the results of Section~\ref{sect:SBC}.
As in Section~\ref{sect:reduction},
there exists a vector bundle $W$ on $C$ such that
\begin{itemize}
\item{
$W^{\vee} \simeq V\otimes \iota^*V$, which yields an 
involution 
$$\begin{array}{rcl}
W^{\vee}&\stackrel{i}{\ra} & W^{\vee} \\
\downarrow &               & \downarrow \\
   C       & \stackrel{\iota}{\ra}    & C
\end{array}
$$
given by $i(v_c\otimes v'_{\iota(c)})=v'_{\iota(c)}\otimes v_{c}$.}
\item{$W$ arises as an extension (see \ref{eqn:ET}))
$$0 \ra W \ra g^*\pi_*H \ra Q \ra 0,$$
where $Q$ is a skyscraper sheaf supported at the singularities
of the fibers of $\cX\times_B C \ra C$.}
\end{itemize}
The extension above dualizes to 
$$0 \ra g^*(\pi_*H)^{\vee} \ra W^{\vee} \ra Q' \ra 0,$$
where $Q'$ is also a skyscraper sheaf supported where the fibration fails
to be smooth.

Now $V$ is stable provided $W^{\vee}$ is stable, or even if 
$(W^{\vee},i)$ is stable as a bundle with involution, i.e., we only test
against quotient bundles compatible with the action of $i$.
Bhosle \cite[\S 1]{BhosleMA} shows that this is equivalent to the following 
form of {\em parabolic semistability} for $(\pi_*H)^{\vee}$:  For every 
isotropic subbundle 
$F \subset (\pi_*H)^{\vee}$  we have
$$\frac{\deg(\pi_*H)^{\vee}+\frac{1}{2}\Delta}{4} \ge 
\frac{\deg(F)+\frac{1}{2}\#\{x_j \in \bP(F)\}}{r}$$
where $r=\rank(F)$.  
For stability, we impose strict inequality. 
(There appears to be a notational inconsistency between
Definition 1.1 and Proposition 1.2 of \cite{BhosleMA};  here we rely on the proof
of the proposition.)

If the covering $C\ra B$ is non-trivial then $\cF \ra B$ has no sections,
i.e., there are no isotropic subbundles or rank two.
Thus isotropic subbundles $F\neq 0$ have rank one, so 
$\bP(F)=\sigma(B)$ for some section $\sigma:B\ra \cX$.  And if $\pi:\cX \ra B$
has square-free discriminant, sections avoid singularities in the fibers.  
Then the condition takes the form
\begin{equation} \label{eqn:stability}
\deg((\pi_*H)^{\vee})+\frac{1}{2}\Delta \ge
4\deg(F).
\end{equation}

We can express this in terms of the heights of the fibrations,
using the formulas of Section~\ref{sect:QFDH}.  
We may normalize $(\pi_*H)^{\vee}=E$
where 
$$
\Delta = -2\deg(E)+4\deg(I), 
\quad \deg(I)= \begin{cases} 0 & \text{ if } \epsilon(\pi)\equiv 0\pmod{2} \\
                            1 & \text{ if } \epsilon(\pi)\equiv 1\pmod{2}
		\end{cases}.
$$
Let $\psi:\bP(E) \ra B$ be the structure map and
$\xi=c_1(\cO_{\bP(E)}(1))$ the relative hyperplane class;
it follows that
$$c_1(T_{\bP(E)/B})=c_1(E)+4\xi.$$
The quadratic form defining $\cX$ is a symmetric homomorphism
$q:E \ra E^{\vee}\otimes I$, so 
$$[\cX]=2\xi+c_1(I).$$
The standard exact sequence
$$0 \ra T_{\cX/B} \ra T_{\bP(E)/B}|\cX \ra N_{\cX/\bP(E)} \ra 0$$
implies
$$c_1(T_{\cX/B})=c_1(E)+2\xi - c_1(I).$$
Pulling back via $\sigma$, we obtain
$$h_{\omega_{\pi}^{-1}}(\sigma)=\deg(E)+2\deg(\sigma^*\xi)-\deg(I)=
-\frac{1}{4}\Delta+\frac{1}{2}\deg(E)+2\deg(\sigma^*\xi).$$
On the other hand, we can express
$$F=\sigma^*\cO_{\bP(E)}(-1),$$
as the latter is the tautological subbundle for $\bP(E)$.  
Thus Equation~\ref{eqn:stability} translates into
$$
\deg(E)+\frac{1}{2}\Delta \ge -4\deg \sigma^*\xi=-2h_{\omega_{\pi}^{-1}}(\sigma)-\frac{1}{2}\Delta+\deg(E)
$$
which simplifies to
$$h_{\omega_{\pi}^{-1}}(\sigma) \ge -\frac{\Delta}{2}.$$

We summarize this computation:
\begin{prop} \label{prop:makestable}
Let $\pi:\cX \ra B$ be a quadric surface fibration with square-free discriminant 
of degree $\Delta$.
Assume that
\begin{itemize}
\item{the discriminant double cover $C\ra B$ is non-trivial;}
\item{all sections $\sigma:B\ra \cX$ satisfy the inequality
$$h_{\omega_{\pi}^{-1}}(\sigma) \ge -\frac{\Delta}{2}.$$}
\end{itemize}
Then the Fano variety of lines $\cF \ra C$ is the projectivization of a semistable
vector bundle.
\end{prop}
Note that the first condition holds provided $\Pic(X)\simeq \bZ$, i.e.,
the monodromy (or Galois) action exchanges the two rulings of the geometric generic
fiber of $\pi$.  
\begin{rema} \label{rema:bound}
The inequality can be verified effectively.  Since $\omega_{\pi}^{-1}$ is ample
relative to $\pi$, there exists a line bundle $A$
on $B$ such that $\omega_{\pi}^{-1}\otimes A$ is ample on $\cX$.  The sections
violating this condition have degree less than $\deg(A)-\frac{\Delta}{2}$,
and thus are bounded.
\end{rema}

\section{Arithmetic applications}
\label{sect:AA}
From now on, we assume $k$ is a finite field of odd characteristic
and $B\simeq \bP^1$.  We retain the notation of Sections~\ref{sect:reduction} and
\ref{sect:casebycase}, so
$\cX \ra \bP^1$ is a quadric surface fibration with square-free discriminant of 
degree $\Delta=2\sg+2$.
Let $C$ be the discriminant curve and $\cF\ra C$ the Fano variety of lines.
The geometric analysis of Section~\ref{sect:reduction} applies, giving
a bijection between
sections $\sigma:\bP^1 \ra \cX$ and 
sections $\tau:C \ra \cF$.  

\subsection*{Effective existence results}
\begin{prop}
There exists a section $\sigma:\bP^1 \ra \cX$ defined over $k$ with
\begin{equation} \label{eqn:htbd}
h_{\omega_{\pi}^{-1}}(\sigma) \le \begin{cases} \frac{\Delta}{2}-2 & \text{ if } \epsilon(\pi)\equiv 0 \pmod{2} \\
					   \frac{\Delta}{2}-3 & \text{ if } \epsilon(\pi)\equiv 1 \pmod{2}.
			\end{cases}
\end{equation}
\end{prop}
\begin{proof}
Recall Equations~\ref{eqn:parity} and \ref{eqn:hV}
$$\deg(V)\equiv \epsilon(\pi)\pmod{2}, \quad
h_{\omega_{\pi}^{-1}}(\sigma)=\deg(V\otimes L)-\Delta/2.$$
The morphism $\lambda:\Sect(\cF/C)\ra \Pic(C)$ is dominant over $\Pic^d(C)$
provided $\Gamma(V\otimes L)\neq 0$ for generic $L$ of degree $d$.  
This is guaranteed to be the case if $\chi(V\otimes L)>0$;  our hypothesis implies
$\deg(V\otimes L) \ge 2\sg-1$, which yields the necessary positivity.  

The generic fiber of $\lambda$ consists of a non-empty open subspace of the projective space
$\bP(\Gamma(V\otimes L))$.
Let $\overline{\Sect(\cF/C)}$ denote the closure 
of $\Sect(\cF/C)$ in the Hilbert scheme parametrizing divisors in $\cF$;
$\lambda$ extends to $\overline{\Sect(\cF/C)}$.
Its fibers are projective spaces $\bP(\Gamma(V\otimes L))$ parametrizing linear
series on $\cF=\bP(V)$, whose members are `broken sections',
consisting of one section of $\cF \ra C$ together with
a configuration of fibers.  

Since $\Pic^d(C)$ is a principal homogeneous space over an abelian variety, 
Lang's Theorem implies $\Pic^d(C)(k)\neq \emptyset$.  Consequently
$$\overline{\Sect(\cF/C)}(k)\neq \emptyset$$
corresponding to a broken section of height bounded by (\ref{eqn:htbd});  
the (unique) horizontal component satisfies the same inequality.
\end{proof}

\subsection*{Effective weak approximation}
\begin{theo}
Let $\pi:\cX \ra \bP^1$ be a quadric surface fibration with
square-free discriminant of degree $\Delta$,
defined over a finite field $k$.  Assume that
\begin{itemize}
\item{the discriminant double cover $C \ra \bP^1$ is non-trivial
over $\bar{k}$;}
\item{all sections $\sigma:\bP^1 \ra \cX$ over $\bar{k}$ satisfy
$$h_{\omega^{-1}_{\pi}}(\sigma) \ge -\Delta/2.$$}
\end{itemize}
Fix a positive integer $N$, distinct geometric points
$b_1,\ldots,b_N \in \bP^1 \setminus \md$
and $x_j \in \cX_{b_j}=\pi^{-1}(b_j)$ for $j=1,\ldots,N$, such
that $\{x_1,\ldots,x_N\}$ is defined over $k$.   Then 
there exists a section $\sigma:\bP^1\ra \cX$ satisfying
\begin{itemize}
\item{$\sigma(b_j)=x_j$ for $j=1,\ldots,N$;}
\item{$h_{\omega_{\pi}^{-1}}(\sigma) \le \frac{3}{2}\Delta+2N$.}
\end{itemize}
\end{theo}
\begin{proof}
Let $g:C\ra \bP^1$ denote the discriminant cover and write
$g^{-1}(b_j)=\{c'_j,c''_j\}$;  let $R_j' \subset V_{c'_j}$ and
$R''_j \subset V_{c''_j}$ denote the one-dimensional subspaces
corresponding to the lines of $\cX_{b_j}$ containing $x_j$.  We seek a section
$\tau:C\ra \cF$ such that $\tau(c'_j)\in \bP(R'_j)$ and 
$\tau(c''_j) \in \bP(R''_j)$.  This imposes $2N$ independent
linear conditions
on the sections of $V$, which together are defined over $k$.

Proposition~\ref{prop:makestable} implies $\cF=\bP(V)$ for $V$
semistable.  Corollary~\ref{coro:gg} gives the existence of 
$t\in \Gamma(C,V)$ satisfying
$$0 \neq t(c'_j) \in R'_j, \quad 0 \neq t(c''_j) \in R''_j$$
provided 
$$\deg(V) \ge 4\sg-1+(2N-1).$$  
Since 
$$h_{\omega_{\pi}^{-1}}(\sigma)=\deg(V)-\Delta/2=\deg(V)-(\sg+1)$$
we obtain
$$h_{\omega_{\pi}^{-1}}(\sigma) = 3\sg-3+2N=\frac{3}{2}\Delta+2N.$$
\end{proof}
\begin{rema}
Our argument yields two variants:
\begin{itemize}
\item{we can approximate any collection of jet data over
places of good reduction, defined over $k$,
with lengths summing to $N$;}
\item{if
the ground field $k$ is algebraically closed we can improve the bound
to 
$$h_{\omega_{\pi}^{-1}}(\sigma) \le \Delta+2N-2$$
by tensoring $V$ with a generic $L \in \Pic(C)$.}
\end{itemize}
\end{rema}

\bibliographystyle{amsalpha}
\bibliography{quadricbundle}

\end{document}